\documentclass[11pt,fleqn]{amsart} 

\usepackage[usenames,dvipsnames,condensed]{xcolor}

\usepackage{amsthm}
\usepackage{amsfonts}
\usepackage[english]{babel}
\usepackage[usenames]{xcolor}
\usepackage{graphicx}
\usepackage{soul}
\usepackage{stfloats}
\usepackage{morefloats}
\usepackage{cite}
\usepackage{lscape}
\usepackage{epstopdf}
\usepackage{braket}
\usepackage[lite]{amsrefs}
\usepackage{mathbbol}
\usepackage{tikz}
\usepackage{shuffle}
\usepackage{amsmath}
\usepackage{tikz-cd}

\usepackage{amsmath,amstext,amsopn,amsfonts,eucal,amssymb}
\usepackage{graphicx,wrapfig,url}

\newtheorem{theorem}{Theorem}[section]
\newtheorem{corollary}[theorem]{Corollary}
\newtheorem{lemma}[theorem]{Lemma}

\newtheorem{definition}[theorem]{Definition}

\newtheorem{example}[theorem]{Example}

\newtheorem{remark}[theorem]{Remark}

\begin{document}

\title[Perturbative YBO]{Perturbative Expansion of Yang-Baxter Operators}

\author{Emanuele Zappala} 
\address{Department of Mathematics and Statistics, Idaho State University\\
	Physical Science Complex |  921 S. 8th Ave., Stop 8085 | Pocatello, ID 83209} 
\email{emanuelezappala@isu.edu}

\maketitle

\begin{abstract}
		We study the deformations of a wide class of Yang-Baxter (YB) operators arising from Lie algebras. We relate the higher order deformations of YB operators to Lie algebra deformations. We show that the obstruction to integrating deformations of self-distributive (SD) objects lie in the corresponding Lie algebra third cohomology group, and interpret this result in terms of YB deformations. We show that there are YB operators that admit integrable deformations (i.e. that can be deformed infinitely many times), and that therefore give rise to a full perturbative series in the deformation parameter $\hbar$. We consider the second cohomology group of YB operators corresponding to certain types of Lie algebras, and show that this can be nontrivial even if the Lie algebra is rigid, providing examples of nontrivial YB deformations that do not arise from SD deformations.
\end{abstract}

\date{\empty}

\tableofcontents

\section{Introduction}

The study of Yang-Baxter (YB) operators has been initiated in statistical mechanics \cite{Jim}, where they were used for the conservation of momentum in scattering processes. After their introduction in statistical mechanics, YB operators have found deep applications in geometric topology \cite{Tur,Oht}, where the YB equation is the algebraic counterpart of the combinatorial Reidemeister move III, while the invertibility of YB operators takes the diagrammatic form of Reidemeister move II. This important relation between combinatorial moves that characterize the isotopy classes of links, and the (algebraic) operatorial YB equation paved the way to the construction of quantum invariants of links, as traces of operators derived from YB operators.

 Furthermore, YB operators play a fundamental role in the theory of Ribbon Hopf algebras \cite{CP_quantum}. The latter are used (with certain additional assumptions) to construct invariants of $3$-manifolds, and as such they are widely employed in theoretical physics -- e.g. in Chern-Simons theory \cite{Hen}, and quantum computation -- e.g. in the Freedman-Kitaev-Wang model \cite{FKW}. 
 
 Homology theories of YB operators (in the set-theoretic context) were introduced by Carter, Elhamdadi and Saito in \cite{CCES1}. The inspiring construction for the theory was a procedure of generalizing the construction of knot invariants from quandles and racks, which are known algebraic objects that naturally produce (set-theoretic) YB operators. Since their introduction, such theories have been studied by various authors -- e.g. \cite{PW,LV,ESZ}. Moreover, related cohomology theories of YB operators have been considered since the introduction of YB homology -- see \cite{Eis1,Eis2,YB_deform}.
 
 Self-distributive (SD) objects in tensor categories are a generalization of quandles and racks to the setting of general tensor categories, instead of the category of sets -- see \cite{3Lie_AZ,CCES,Leb}. SD objects give rise to YB operators in tensor categories, and can therefore be used to define quantum invariants of links when a suitable notion of trace is present in the category \cite{EZ}. As such, one is interested in studying the YB operators that arise from SD structures in tensor categories and, of course, modules and vector spaces are a very natural starting point. Toward the objective of constructing quantum invariants of links, one is interested in modifying the YB operators through some cohomological procedure, mostly in view of the success of the cohomological invariants of \cite{CJKLS}.

In order to modify YB operators associated to SD structures for quantum invariants, one can proceed in two very natural ways. Firstly, one can use the cohomology of SD structures with coefficients in the unit object of the category, which is a generalization of the standard cohomology theory of racks and quandles in \cite{CJKLS}. The second approch, is computationally more involved (in terms of cohomology), but it has the convenience of being general for YB operators, without the need of these arising from SD structures. Namely, this is the approach of deforming a YB operator through nontrivial classes of its second cohomology group, and use this to obtain quantum link invariants. This can be performed through a cohomological infinitesimal deformation in the style of Gerstenhaber \cite{Gerst}. 

The first approach has been used in \cite{Grana,EZ}. In \cite{EZ} it was shown to have nontrivial applications for Hopf algebras and, more generally, Hopf monoids in tensor categories. However, another well known class of SD structures arises from ($n$-ary) Lie algebras \cite{CCES,3Lie_AZ}. These structures were not considered in \cite{EZ}, and it is not clear whether nontrivial quantum invariants arise from them using cohomology with coefficients in the ground ring of the Lie algebra. This has prompted the study of SD deformations of SD structures arising from Lie algebras in \cite{YB_deform}. Under some mild assumptions it was shown that Lie second cohomology and SD second cohomology are isomorphic, and that the former always injects into the latter. It was also shown that such deformations naturally induce YB deformations of the associated YB operator. Hower, YB operator deformations might in principle arise even if they are not induced by Lie algebra deformations. 

The scope of this article is two-fold. First, we study higher (i.e. not infinitesimal) deformations of YB operators associated to Lie algebras. As a result we obtain YB operators that are perturbed with respect to the deformation parameter $\hbar$. Such construction would produce quantum invariants in the form of polynomials or power series in $\hbar$, and this will be studied in the future. 
Second, we study YB deformations that do not arise from Lie deformations, and that are in a sense purely YB deformations. We find that even simple Lie algebras, whose rigidity has been long  known, give rise to YB operators that can be nontrivially deformed. 

This article is structured as follows. In Section~\ref{sec:pre} we recall some definitions that are used throughout the article, and give the references for the needed background material. In Section~\ref{sec:higher_quandle} we consider higher deformations of SD structures in relation to Lie algebra deformations. In Section~\ref{sec:higher_YB} we consider higher deformations of YB operators and study their relation to Lie algebra higher deformations. In Section~\ref{sec:coh} we give a characterization of the second cohomology group for YB operators arising from a class of Lie algebras. In Section~\ref{sec:ex} we give some examples of YB operators that can be deformed infinitely many times, and therefore give rise to a full perturbative series in the deformation parameter $\hbar$, and show that rigid Lie algebras can give rise to nontrivially deformable YB operators.

\section{Preliminaries}\label{sec:pre}
 
We collect here some notions that will be used throughout the article.

\subsection{Racks and quandles}

First recall that a {\it rack} is a magma $X$ with a binary operation $*$ satisfying the two conditions:
\begin{itemize}
	\item 	
			Right multiplication $R_y : X\longrightarrow X$ is invertible for all $y\in X$, where $R_y(x) := x*y$. 
	\item 
			$(x*y)*z = (x*z)*(y*z)$, called self-distributivity.
\end{itemize}
These properties are algebraic versions of the Redemeister moves II and III, respectively. If, in addition, a rack satisfies the idempotency relation, $x*x = x$, then it is called a {\it quandle}. See \cite{CJKLS}.
	Let $X$ be a rack. Define chain groups $C_n(X)$ to be the free abelian group generated by the elements of $X^n$ for each $n$. Then, in \cite{CJKLS} the $n^{\rm th}$-differential $\partial_n$ was defined on generators according to the assignment
\begin{eqnarray*}
	\lefteqn{\partial_n (x_1,\ldots , x_n)}\\
	&=& \sum_{i=2}^n (-1)^n[(x_1,\ldots ,x_{i-1},\hat x_i, x_{i+1}, \dots , x_n)\\
	&& \hspace{0.5cm} - (x_1*x_i, \ldots, x_{i-1}*x_i, \hat x_i,x_{i+1}, \ldots , x_n)].
\end{eqnarray*}
These differentials define a homology theory for quandles, and its dualization gives a cohomology which has been extensively used in geometric topology to construct state-sum invariants of links and knotted surfaces \cite{CJKLS}.

The categorical counterpart of racks and quandles was introduced in \cite{CCES}. They are self-dsitributive (SD) objects in tensor categories, defined through a relatively straightforward generalization of the set-theoretic definitions above. For a general perspective on categorical SD structures with $n$-ary operators see Section~7 in \cite{EZ}. The cohomological theory that classifies the infinitesimal deformations of SD structures was introduced and studied in \cite{YB_deform}. We very briefly recall some of the definitions from \cite{EZ} and \cite{YB_deform}. 

Let $\mathcal C$ be a symmetric monoidal category and let $X$ be an object in $\mathcal C$. Then, we say that a coalgebra object $(X,\Delta)$ is an SD object if there exists a morphism (of coalgebras) $q: X\otimes X \longrightarrow X$ that makes the following diagram commute:
\begin{center}
	\begin{tikzcd}
		& X^{\otimes 4}\arrow[dl,"\shuffle",swap] &  & X^{\otimes 3}\arrow[dr,"q\otimes \mathbb 1"] \arrow[ll,"\mathbb 1^{\otimes 2}\otimes \Delta",swap]&\\
		X^{\otimes 4}\arrow[d,"q\otimes q",swap]& & & & X^{\otimes 2}\arrow[d,"q"]\\
		X^{\otimes 2}\arrow[rrrr,"q",swap]& & & &X,
	\end{tikzcd} 
\end{center}
where $\shuffle = \mathbb 1\otimes \tau \otimes \mathbb 1$, having indicated by $\tau$ the switching morphism of  $\mathcal C$, $\tau_{X,X} : X\otimes X \longrightarrow X\otimes X$. If $\mathcal C$ is the category of sets with $\otimes  = \times$, and $X$ is a rack with $\Delta(x) = x\times x$, then one recovers the usual definition of rack given above. More general examples arise from Hopf monoids in symmetric monoidal categories, see \cite{EZ}. In the category of modules, an SD object is a coalgebra $(X,\Delta)$ endowed with a coalgebra homomorphism $q: X\otimes X \longrightarrow X$ satisfying the condition
\begin{eqnarray*}
		q(q(x\otimes y)\otimes z) = q(q(x\otimes z^{(1)})\otimes q(y\otimes z^{(2)})),
\end{eqnarray*}
where we have used Sweedler's notation to indicate the coproduct $\Delta(z) = z^{(1)} \otimes z^{(2)}$. We say that $q$ is invertible if there exists $\tilde q$, that turns $(X,\Delta)$ into an SD object as well, such that $\tilde q(q\otimes \mathbb 1)(\mathbb 1\otimes \Delta) = q(\tilde q\otimes \mathbb 1)(\mathbb 1\otimes \Delta) = \mathbb 1 \otimes \epsilon$, where $\epsilon$ is the counit morphism of the coalgebra object $X$. In this situation we say that $X$ is a rack object. 
 The SD object arising from Lie algebras considered in this article are all invertible. 
 
There is a similar notion of $n$-ary SD object and rack object as well. The definitions are obtained by straightforward generalization of the notions given above for the binary case. Details can be found in \cite{EZ,YB_deform}. 

For an SD object $(X,\Delta)$ in the category of $\mathbb k$-modules, we can define the SD cohomology group (with coefficients in $X$) that classifies SD deformations of $X$ following \cite{YB_deform}. We set $C^1_{\rm SD}(X,X)$ to be the module of coderivations of $(X,\Delta)$. Next, define the second cochain group $C^2_{\rm SD}(X,X)$ to be the module of linear maps $\psi: X\otimes X \longrightarrow X$ satisfying the property $\Delta \psi = (\psi\otimes q + q\otimes \psi) (\mathbb 1\otimes \tau \otimes \mathbb 1)\Delta \otimes \Delta$. We set $C^3_{\rm SD}(X,X) = {\rm Hom}_{\mathbb k}(X^{\otimes 3},X)$, consisting of module homomorphisms.  The first differential $\delta^1$ is defined as
\begin{eqnarray*}
		\delta^1 f(x\otimes y) = f(q(x\otimes y)) - q(f(x)\otimes y) - q(x\otimes q(y)),
\end{eqnarray*}
while the second differential is defined as
\begin{eqnarray*}
	\delta^2 \psi (x\otimes y \otimes z) &=& q(\psi(x\otimes y)\otimes z) + \psi(q(x\otimes y) \otimes z) - \psi(q(x\otimes z^{(1)})\otimes q(y\otimes z^{(2)})\\
	&& - q(\psi(x\otimes z^{(1)})\otimes q(y\otimes z^{(2)}) - q(q(x\otimes z^{(1)})\otimes \psi(y\otimes z^{(2)}).
\end{eqnarray*}

We recall that given a Lie algebra $\mathfrak g$ over the ring $\mathbb k$, we obtain a rack object as follows \cite{CCES}. Set $X = \mathbb k \oplus \mathfrak g$ and indicate elements of $X$ as pairs $(a,x)$. The coalgebra structure on $X$ is given by $\Delta (a,x) = (a,x) \otimes (1,0) + (1,0) \otimes (0,x)$. The coalgebra counit is defined es $\epsilon (a,x) = a$. Define $q((a,x)\otimes (b,y)) = (ab, bx + [x,y])$. It can be shown that $(X,\Delta)$ is a rack object with this choice of $q$. When $\mathfrak g$ is an $n$-Lie algebra, a similar construction, mutatis mutandis gives an $n$-rack object. 

On $X = \mathbb k \oplus \mathfrak g$ we define the projections $\pi_0: X \longrightarrow \mathbb k$ and $\pi_1: X \longrightarrow \mathfrak g$, as $\pi_0(a,x) = a$ and $\pi_1(a,x) = x$. We will also often use the tensor products $\pi_i\otimes \pi_j$.

\subsection{Yang-Baxter operators}

Let $V$ be a vector space or module, and let $R: V\otimes V \longrightarrow V\otimes V$ be a linear map. If $R$ satisfies the operatorial equation
\begin{eqnarray}\label{eqn:YBE}
		(R\otimes \mathbb 1)(\mathbb 1\otimes R)(R\otimes \mathbb 1) = (\mathbb 1\otimes R)(R\otimes \mathbb 1)(\mathbb1 \otimes R),
\end{eqnarray}
then we say that $R$ is a pre-Yang-Baxter operator. An invertible pre-Yang-Baxter operator is called Yang-Baxter (YB) operator. The full cochain complex for YB cohomology can be found in \cite{Eis1}. We will, however, mostly use the low dimensional differentials, since we are mostly concerned with the second YB cohomology group, and recall them here:  
\begin{eqnarray*}
	\delta^1_{\rm YB} (f) 
	&=&
	R ( f \otimes {\mathbb 1}) + R ( {\mathbb 1} \otimes  f ) 
	-  ( f \otimes {\mathbb 1}) R - ( {\mathbb 1} \otimes  f )  R ,\\
	\delta^2_{\rm YB} (\phi) &=&
	(R \otimes {\mathbb 1} ) ( {\mathbb 1}  \otimes R ) ( \phi \otimes  {\mathbb 1} )
	+ (R \otimes {\mathbb 1} ) ( {\mathbb 1}  \otimes \phi ) ( R \otimes  {\mathbb 1} )
	+  (\phi \otimes {\mathbb 1} ) ( {\mathbb 1}  \otimes R ) ( R \otimes  {\mathbb 1} ) \\
	& & -  ( {\mathbb 1}  \otimes R ) ( R \otimes  {\mathbb 1} ) ( {\mathbb 1}  \otimes \phi ) 
	- ( {\mathbb 1}  \otimes R ) ( \phi \otimes  {\mathbb 1} ) ( {\mathbb 1}  \otimes  R ) 
	- ( {\mathbb 1}  \otimes \phi ) ( R \otimes  {\mathbb 1} ) ( {\mathbb 1}  \otimes  R ),
\end{eqnarray*}
where $f : V \longrightarrow V$ is a YB $1$-cochain, and $\phi: V\otimes V \longrightarrow V$ is a YB $2$-cochain.

We recall that given a rack object $X$, we obtain a YB operator by setting
\begin{eqnarray*}
		R(x\otimes y) = y^{(1)} \otimes q(x\otimes y^{(2)}).
\end{eqnarray*}
This procedure generalizes the standard set-theroetic YB operator arising from racks. Given an $n$-rack, we can similarly obtain a YB operator on $X^{\otimes (n-1)}$ by setting
\begin{eqnarray*}
	R(x_1\otimes x_2 \otimes \cdots x_n) = x_2^{(1)}\otimes \cdots x_n^{(1)} \otimes q(x_1\otimes x_2^{(2)}\otimes \cdots x_n^{(2)}).
\end{eqnarray*} 

 This correspondence has been used in \cite{YB_deform} to show that from an $n$-Lie algebra, and its corresponding $n$-rack object $X$, one can derive a monomorphism between the second rack cohomology group into the YB second cohomology group. In other words, infinitesimal deformations of $n$-Lie algebras (or infinitesimal deformations of the corresponding $n$-rack objects) give rise to infinitesimal deformations of the corresponding YB operators. Moreover, if the Lie deformation is nontrivial, the corresponding YB deformation is nontrivial as well.

\section{Higher deformations of rack objects}\label{sec:higher_quandle}

Let us consider an $n$-ary rack object $(X,T,\Delta)$ induced by an $n$-ary Lie algebra $\frak g$ as in \cite{3Lie_AZ,YB_deform}. In this case, in \cite{YB_deform} it was shown that there is an injective map from the Lie algebra second cohomology group of $\frak g$ to the rack second cohomology group of $X$. Moreover, under some mild assumptions this homomorphism is also an isomorphism. As the second rack cohomology group $H^2_{\rm SD}(X,X)$ characterizes the infinitesimal deformations of $X$, we have an infinitesimal deformation of $X$ corresponding to each $n$-Lie algebra $2$-cocycle.  We want to consider the obstructions to extending this infinitesimal deformation to higher orders. 

Throughout this article we leave the comultiplication undeformed. In \cite{YB_deform} it was shown that SD deformations where the comultiplication is deformed nontrivially exist. However, we leave this more general case to a subsequent study.

We recall that, following definitions of \cite{YB_deform},  a $2$-cochain $\psi: X\otimes X\longrightarrow X$ is called special if for all $(a,x)$ and $(b,y)$ in $X$ it holds $\psi((a,x)\otimes (b,y)) = (0,\phi(x\otimes y))$ for some $\phi : \mathfrak g\otimes \mathfrak g \longrightarrow \mathfrak g$.

For a Lie $k$-cochain $\phi$, define the following SD $k$-cochain, $\Theta^k(\phi)((a_1,x_1)\otimes \cdots \otimes (a_k,x_k)) = (0,\phi(x_1\otimes \cdots \otimes x_k))$. This correspondence defines a map $C^k_{\rm Lie}(\mathfrak g,\mathfrak g) \longrightarrow C^k_{\rm SD}(X,X)$, from the $k^{\rm th}$ Lie cochain group  to the $k^{\rm th}$ SD cochain group.

\begin{theorem}\label{thm:quandle_deform}
	Let $\frak g$ be an $n$-Lie algebra and let $(X,T,\Delta)$ denote its corresponding $n$-rack object. Assume that $\phi = \sum_{i=0}^m \hbar^i\phi_i$ is an order $m$ deformation of the bracket of $\frak g$. Then,  the correspondence $\Theta^{m+1}$ gives an order $m+1$ deformation of $X$ if the obstruction to extending $\phi$ to an order $m+1$ Lie deformation vanishes.  Moreover, if $\frak g$ has trivial center, then the obstructions to integrating any infinitesimal deformation to a degree $m$ deformation lie in $H^3_{\rm Lie}(\frak g, \frak g)$, for all $m$. 
\end{theorem} 
\begin{proof}
	We proceed by induction on the order $m$ of the deformation. For the base of induction, i.e. $m=1$, we use Theorem~5.1 and Proposition~6.6 in \cite{YB_deform}. In fact, let $\phi = \phi_0 + \hbar \phi_1$, where $\phi_0 = [\bullet, \bullet]$ is the bracket of $\frak g$. Then, it is known that $\phi$ is a deformation if and only if $\phi_1$ is a Lie algebra $2$-cocycle of $\frak g$. Let us define the map $\psi = \psi_0 + \hbar \psi_1$ where $\psi_0 := T$ is the $n$-rack operation corresponding to $\phi_0$, and $\psi_1((a_1,x_1)\otimes \cdots \otimes (a_n,x_n)) = (0, \phi_1(x_1\otimes \cdots \otimes x_n))$. Using Proposition~6.6 in \cite{YB_deform} it follows that $\psi_1$ is an $n$-rack $2$-cocycle and, using Theorem~5.1 in \cite{YB_deform}, $\psi$ is an infinitesimal deformation of $T$. Now, assume that $\phi$ can be deformed through a map $\phi_2$, and let us consider a deformation of order $2$ of the $n$-rack structure $\hat \psi = \psi + \hbar^2\psi_2$, where $\psi_2 := \Theta^3(\phi_2)$. Then, imposing the $n$-ary self-distributive condition on $\hat \psi$, we only need to verify that terms of degree $2$ hold, since terms in degree higher than $2$ in $\hbar$ vanish modulo $\hbar^3$, and terms of lower degree are already known to be satisfied by hypothesis. Moreover, from the definition of $\Theta^3$, we have that $\psi_2(X^{\otimes n}) \subset \frak g^{\otimes n}$, and also $\psi_2$ maps vectors containing a tensorand in a copy of the ground ring $\mathbb k$ to zero. These facts simplify the computation very much. For notational simplicity, we write the equations in the case of ternary Lie algebras and $3$-rack. This is generalized immediately to the $n$-ary case.
	For the LHS of $n$-ary self-distributivity evaluated on simple tensors we have
	\begin{eqnarray*}
	\lefteqn{	
	\hat \psi(\hat \psi ((a_1,x_1)\otimes (a_2,x_2) \otimes (a_3,x_3))\otimes (a_3,x_3)\otimes (a_4,x_4) \otimes (a_5,x_5))}\\
	 &=& \hbar^0 [\cdots] + \hbar^1[\cdots] + \hbar^2[\phi_2(\phi_0(x_1\otimes x_2 \otimes x_3)\otimes x_4 \otimes x_5) \\
	 &&+ \phi_1(\phi_1(x_1\otimes x_2 \otimes x_3)\otimes x_4 \otimes x_5)+ \phi_0(\phi_2(x_1\otimes x_2 \otimes x_3)\otimes x_4 \otimes x_5)] + \hbar^3[\cdots].
	\end{eqnarray*}
	For the RHS of the equation we have, setting for simplicity $y_i :=  x_4^{(i)} \otimes x_5^{(i)}$ (observe the superscript due to comultiplication in Sweedler notation)
	\begin{eqnarray*}
	\lefteqn{\hat \psi(\hat \psi(x_1\otimes y_1)\otimes \hat \psi(x_2\otimes y_2) \otimes \hat \psi(x_3\otimes y_3))}\\
	 &=& \hbar^0[\cdots] + \hbar^1[\cdots] + \hbar^2[\phi_2(\phi_0\otimes \phi_0\otimes \phi_0)
	+ \phi_0(\phi_2\otimes \phi_0\otimes \phi_0) + \phi_0(\phi_0\otimes \phi_2\otimes \phi_0) \\
	&&+ \phi_0(\phi_0\otimes \phi_0\otimes \phi_2) + \phi_1(\phi_1\otimes \phi_0\otimes \phi_0) + \phi_1(\phi_0\otimes \phi_1\otimes \phi_0) + \phi_1(\phi_0\otimes \phi_0\otimes \phi_1) + \\
	&&+ \phi_0(\phi_1\otimes \phi_1\otimes \phi_0) + \phi_0(\phi_1\otimes \phi_0\otimes \phi_1) + \phi_0(\phi_0\otimes \phi_1\otimes \phi_1)](x_1\otimes y_1\otimes x_2\otimes y_2\otimes x_3\otimes y_3) \\
	&& + \hbar^3[\cdots].
	\end{eqnarray*}
	Equating, we see that in order for $\hat \psi$ to satisfy the self-distributive property, the following equation
	\begin{eqnarray}\label{eqn:def_deg_2}
	\delta^2_{\rm Lie}(\psi_2) + \frac{1}{2}[\psi_1,\psi_1]_{\rm NR} = 0,
	\end{eqnarray}
	needs to hold, where $[\bullet, \bullet]_{\rm NR}$ is the ternary ($n$-ary in general) version of the Lie bracket defined by Nijenhuis and Richardson in \cite{NR} in the binary case. Therefore, this is the same obstruction to extend $\phi_0 + \hbar \phi_1$ to a deformation of degree $2$, and the obstruction lies in $H_{\rm Lie}^3(\frak g, \frak g)$ generalizing the binary case shown in \cite{NR}. To complete the proof of the base of induction, we need to show that this is always the case when $\frak g$ has trivial center. Observe that, from Lemma~6.3 in \cite{YB_deform}, whenever $\frak g$ has trivial center, the infinitesimal deformation $\psi_1$ of the SD structure is special, and it is characterized by a map $\phi_1: \frak g\otimes \frak g\longrightarrow \frak g$. This in practice means that every infinitesimal deformation of $X$ arises from a deformation of $\frak g$ through $\Theta^2$ (see Theorem~6.4 in \cite{YB_deform}). Now, we want to show that an order $2$ deformation will also be obtained through $\Theta^3$ and a higher order deformation $\phi_2$ for $\frak g$. To do so, observe that the term $[\phi_1,\phi_1]_{\rm NR}$ in Equation~\ref{eqn:def_deg_2} is trivial when evaluated on terms containing a tensorand in $\mathbb k$, because of Lemma~6.3 in \cite{YB_deform}. Therefore, on simple tensors of type $(0,x)\otimes (1,0)\otimes (0,z)$ and $(0,x)\otimes (0,y)\otimes (1,0)$ Equation~\ref{eqn:def_deg_2} reduces to the evaluation of $\delta^2_{\rm Lie}(\psi_2)$. However, this is exactly the same situation of the proof of Lemma~6.3 in \cite{YB_deform}, which can be repeated, showing that $\psi_2 = \Theta^3(\phi_2)$ for some Lie algebra $2$-cochain. Therefore this has been reduced to the previous situation, for which we already know that the obstruction to the degree $2$ deformation lies in $H^3_{\rm Lie}(\frak g,\frak g)$. This completes the proof of the base of induction.
	
	Let us now assume that the statement has been proved for some $m>1$, and let us consider the case $m+1$. Let then $\phi = \sum_{i=0}^m \hbar^i \phi_i$ be an order $m$ deformation of $\frak g$. We want to show that if the obstruction to extending the deformation of $\frak g$ to order $m+1$ vanishes, then we obtain an order $m+1$ deformation of $X$ via $\Theta^{m+1}$. Observe that from the induction step, we already know that $\psi = \Theta^m(\phi)$ is an order $m$ deformation of $X$.  Let us set $\hat \psi = \psi + \hbar^{m+1} \psi_{m+1}$. Let us consider the SD condition for $\hat \psi$. Since we already know that up to degree $m$ the equation is satisfied, we can discard all terms of degree lower than $m+1$. We consider in what follows the binary case, since the $n$-ary case is a straightforward, although cumbersome, generalization of this. For the LHS of the SD condition in degree $m+1$ we obtain
	\begin{eqnarray}\label{eqn:SD_higher_deform_LHS}
\hat \psi(\hat \psi \otimes \mathbb 1) &=& \sum_{i,j=0}^{m+1}\psi_i (\psi_j\otimes \mathbb 1) ,
\end{eqnarray}
For the right hand side of the SD property, we have
\begin{eqnarray}\label{eqn:SD_higher_deform_RHS}
\hat \psi (\hat \psi \otimes \hat \psi)\shuffle (\mathbb 1^{\otimes 2}\otimes \Delta) &=& \sum_{i,j,k=0}^{m+1}\psi_i(\psi_j\otimes \psi_k)\shuffle (\mathbb 1^{\otimes 2}\otimes \Delta)
\end{eqnarray}
Now, using the definition of $\hat \psi$ as the image of $\hat \phi$ through $\Theta^{m+1}$, we can rewrite Equations~\ref{eqn:SD_higher_deform_LHS} and~\ref{eqn:SD_higher_deform_RHS} in terms of $\hat \phi$. For Equation~\ref{eqn:SD_higher_deform_LHS} on simple tensors $(a_1,x_1)\otimes (a_2,x_2)\otimes (a_3,x_3)$ we obtain
\begin{eqnarray*}
	\hat \psi(\hat \psi((a_1,x_1)\otimes (a_2,x_2))\otimes (a_3,x_3)) &=& (a_1a_2a_3, a_2a_3 x + \sum a_3\hbar^i \phi_i(x_1\otimes x_2) + \sum_ja_2\hbar^j \phi_j(x_1\otimes x_3) \\ 
	&& +\sum_{k,\ell}\hbar^{k+\ell} \phi_k(\phi_\ell (x_1\otimes x_2)\otimes x_3) ).
\end{eqnarray*}
For Equation~\ref{eqn:SD_higher_deform_RHS} evaluated on simple tensors as above, and setting $\shuffle (\mathbb 1^{\otimes 2}\otimes \Delta) := \Delta_\shuffle$, we find
\begin{eqnarray*}
\hat \psi (\hat \psi\otimes \hat \psi) \Delta_\shuffle(a_1,x_1)\otimes (a_2,x_2)\otimes (a_3,x_3) &=& (a_1a_2a_3, a_2a_3x_1 + \sum_i a_2\hbar^i \phi_i(x_1\otimes x_3)\\
&& + \sum_j a_3\hbar^j \phi_j(x_1\otimes x_2) + \sum_{k,\ell} \hbar^{k+\ell}[\phi_k(\phi_\ell(x_1\otimes x_3)\otimes x_2)\\
&& + \phi_k(x_1\otimes \phi_\ell(x_2\otimes x_3))]).
\end{eqnarray*} 
From the inductive assumption, up to degree $m$ in the powers of $\hbar$, the SD property holds, so we can restrict ourselves to degrees of order $m+1$ (higher orders vanish modulo $\hbar^{m+2}$). For the LHS of the SD property we obtain 
\begin{eqnarray*}
[\hat \psi(\hat \psi((a_1,x_1)\otimes (a_2,x_2))\otimes (a_3,x_3))]_{{\rm deg}=m+1} &=& a_3\phi_{m+1}(x_1\otimes x_2) + a_2 \phi_{m+1}(x_1\otimes x_3)\\
&& + \phi_{m+1}(\phi_0(x_1\otimes x_2)\otimes x_3) + \phi_0(\phi_{m+1}(x_1\otimes x_2)\otimes x_3)\\
 &&+ \sum_{i,j} \phi_i(\phi_j(x_1\otimes x_2)\otimes x_3)
\end{eqnarray*}
and for the RHS we have
\begin{eqnarray*}
\lefteqn{[\hat \psi (\hat \psi\otimes \hat \psi) \Delta_\shuffle(a_1,x_1)\otimes (a_2,x_2)\otimes (a_3,x_3)]_{{\rm deg}=m+1}}\\
 &=& a_2\phi_{m+1}(x_1\otimes x_3) + a_3\phi_{m+1}(x_1\otimes x_2) + \phi_{m+1}(\phi_0(x_1\otimes x_3)\otimes x_2) + \phi_0(\phi_{m+1}(x_1\otimes x_3)\otimes x_2)\\
&& + \phi_{m+1}(x_1\otimes \phi_0(x_2\otimes x_3)) + \phi_0(x_1\otimes \phi_{m+1}(x_2\otimes x_3))\\
&& + \sum_{i,j} [\phi_i(\phi_j(x_1\otimes x_3)\otimes x_2) + \phi_i(x_1\otimes \phi_j(x_2\otimes x_3))].
\end{eqnarray*}
We observe that all the terms containing $a_i$ cancel out, and what is left can be rewritten as
$$
\delta^2_{\rm Lie} (\phi_{m+1}) + \frac{1}{2}[\phi_{\geq 1},\phi_{\geq 1}]_{\rm NR} = 0,
$$	
which is the obstruction for $\phi$ to be extended to a degree $m+1$ deformation as a Lie algebra \cite{NR}. This completes the inductive step for the first part of the statement. We need to show, now, that when $\frak g$ has trivial center, the deformations in the form of $\Theta^{m+1}(\phi_{m+1})$ are the only ones. Let $\psi = \sum_{i=0}^m\hbar^i \psi_i$ be an order $m$ deformation and set
$\hat \psi = \sum_{i=0}^{m+1} \hbar^i \psi_i$. Then, if $\psi_{m+1}$ is such that $\hat \psi$ is an order $m+1$ deformation, we can equate the LHS and RHS of the SD property at degree $m+1$ substantially following the computation given before, which for the LHS gives
	\begin{eqnarray*}
[\hat \psi(\hat \psi \otimes \mathbb 1)]_{{\rm deg} = m+1} &=& \psi_0 (\psi_{m+1}\otimes \mathbb 1) + \psi_{m+1}(\psi_0\otimes \mathbb 1) + \sum \psi_i(\psi_j\otimes \mathbb 1),
\end{eqnarray*}
where the sum runs over the pairs $(i,j)$ such that $i+j = m+1$ and $i,j\neq m+1$. For the RHS of the SD property, we have
\begin{eqnarray*}
[\hat \psi (\hat \psi \otimes \hat \psi)\shuffle (\mathbb 1^{\otimes 2}\otimes \Delta)]_{{\rm deg} = m+1} &=& \psi_{m+1}(\psi_0\otimes \psi_0)\shuffle (\mathbb 1^{\otimes 2}\otimes \Delta)\\ \notag
&& + \psi_0(\psi_{m+1}\otimes \psi_0 )\shuffle (\mathbb 1^{\otimes 2}\otimes \Delta)\\ \notag
&& + \psi_0(\psi_0\otimes \psi_{m+1} )\shuffle (\mathbb 1^{\otimes 2}\otimes \Delta)\\ \notag
&& + \sum \psi_i(\psi_j\otimes \psi_k )\shuffle (\mathbb 1^{\otimes 2}\otimes \Delta), \notag
\end{eqnarray*}
where the sum runs over $i,j,k\neq m+1$ such that $i+j+k = m+1$. Observe that equating we obtain a term that coincides with $\delta^2_{\rm SD}(\psi_{m+1})$ and a term that only contains terms $\psi_i$ where $i<m+1$, which we will denote $\Omega_{m+1}$. By the inductive assumption, all the maps $\psi_{i<m+1}$ appearing in $\Omega_{m+1}$ are obtained as $\Theta^{i}(\phi_i)$ for some Lie $2$-cochain $\phi_i: \frak g\otimes \frak g\longrightarrow \frak g$. A direct computation shows that $\Omega_{m+1}$, as a result, is a special $2$-cochain mapping $\frak g\otimes \frak g\longrightarrow \frak g$, and being trivial when evaluated on $X_0 := \mathbb k\otimes \mathbb k \oplus \mathbb k \otimes \frak g\oplus \frak g\otimes \mathbb k$. As a consequence, the SD property in degree $m+1$ evaluated on $X_0$ reduces to the $2$-cocycle condition for $\psi_{m+1}$ evaluated on $X_0$. However, from Lemma~6.3 in \cite{YB_deform} we know that this forces $\psi_{m+1}$ to be special, and therefore in the image of $\Theta^{m+1}$. It follows now that if $\phi_{m+1}$ is such that $\Theta^{m+1}(\phi_{m+1}) = \psi_{m+1}$, the obstruction in terms of the Lie algebra $\frak g$ for $\phi_{m+1}$ needs to vanish, and the proof is complete.
\end{proof}

As a direct consequence of the previous result, we obtain the following rigidity criterion for SD structures. 

\begin{corollary}\label{cor:SD_rigid}
	Let $(X,T,\Delta)$ denote the $n$-rack object associated to a semisimple Lie algebra $\frak g$. Then $X$ cannot be deformed as an SD structure.
\end{corollary}
\begin{proof}
	Observe that from Theorem~\ref{thm:quandle_deform} there are no deformations due to terms purely in the Lie algebra, i.e. special, since the second cohomology group of $\frak g$ is trivial. Then the result follows once we show that there are no deformations that are not special. This was shown in \cite{YB_deform}, therefore concluding the proof. 
\end{proof}

\section{Higher deformations of Yang-Baxter operators}\label{sec:higher_YB}

 We now consider the effect of integrating infinitesimal deformations, and we show that the obstruction to lifting YB deformations lies in the $n$-Lie algebra cohomology, under suitable conditions. This construction gives a way of producing higher deformations of YB operators corresponding to an $n$-Lie algebra. 

We start by considering certain types of YB cochains that arise from Lie algebra cochains, and study the obstruction to higher order YB deformations. 

\begin{definition}
	{\rm 
			 For $\phi : \frak g\otimes \frak g \longrightarrow \frak g$ a Lie algebra $2$-cochain with coefficients in $\frak g$, we construct a YB $2$-cochain $\Lambda^2(\phi) : X\otimes X \longrightarrow X\otimes X$ through the assignment 
			\begin{eqnarray*}
				\Lambda^2(\phi)((a,x)\otimes (b,y)) = (b,y)^{(1)} \otimes \phi (x\otimes \pi_1((b,y)^{(2)})),
			\end{eqnarray*}
			where $\pi_1 : X \longrightarrow \frak g$ projects on the second coordinate. 
			We will call these cochains {\it $\Lambda$-cochains}, and the corresponding deformations will be called $\Lambda$-deformations.
	}
\end{definition}
Such a correspondence defines a subclass of YB $2$-cochains, and we will show that the deformation theory of such cochains directly relates to that of the Lie algebra $\frak g$. In the following we assume $\mathbb k$ to have zero characteristic.
 We use the notation found in \cite{NR} for the Lie bracket $[\xi,\chi]$ in the space of alternating maps. We also introduce the following sets
$$
\Gamma_m := \{(i,j,k) \in \mathbb N^{\times 3} \ |\ i+j+k = m,\ i,j,k \neq m \}.
$$
We further introduce the decomposition $\Gamma_m = \sqcup_{l=1,2,3}  \Gamma_m^l \sqcup \hat \Gamma_m$, where $\Gamma_m^l$ is the subset of $\Gamma_m$ such that the triples $(i,j,k)$ have zero in the entry $l$, and $\hat \Gamma_m$ is the subset such that no entry is zero. Observe that from the definition of $\Gamma_m$ at least two entries in $(i,j,k)$ need to be nonzero. 

\begin{theorem}\label{thm:Higher_obstruction}
	Let $\frak g$ be an $n$-Lie algebra, and let $X$ denote the corresponding $n$-ary SD object, with $R$ the induced YB operator. Assume that $\hat R = \sum_{i=0}^m \hbar^i R_i$ is a deformation of order $m$, with $R_0 := R$ and $R_i = \Lambda^2(\phi_i)$ for Lie algebra $2$-cochains $\phi_i$. Then the obstruction to deforming $R$ to degree $m+1$ is given by
	\begin{eqnarray}\label{eqn:YB_obstruction}
	\delta^2_{\rm Lie} \phi_{m+1} + \sum_{k=1}^m \frac{1}{2} [\phi_k,\phi_{m+1-k}] = 0.
	\end{eqnarray}
\end{theorem}
\begin{proof}
		We prove the result for binary SD structure, although the same approach, with notational modifications also give the result for $n$-ary SD structures.
		We proceed by induction on $m$. The case $m=1$ means that we have an infinitesimal deformation of $R$ of type $\Lambda^2(\phi_1)$, and we want to derive the obstruction to lifting this deformation to a quadratic one. From \cite{YB_deform} we already know that $\Lambda^2(\phi_1)$ needs to be a YB $2$-cocycle, and that this fact implies that $\phi_1$ is a Lie algebra $2$-cocycle. Let us consider $\Lambda^2(\phi_2)$, where $\phi_2$ is a Lie $2$-cochain. Here $\phi_0$ indicates the Lie bracket of $\frak g$. Also, observe that $\Lambda^2(\phi_0)$ coincides with the SD operation $q$, as it is seen by a direct computation. For short we will indicate the mapping $\Lambda^2$ simply by $\Lambda$. Let us consider the RHS of the YB equation for the terms which are quadratic in $\hbar$, since we already know that the equation holds for the other terms. Denoting the RHS quadratic terms as $\Psi^R_2$, by evaluating on a simple tensor $(a,x)\otimes (b,y) \otimes (c,z)$ we have
		\begin{eqnarray*}
			\Psi^R_2 (a,x)\otimes (b,y) \otimes (c,z) &=& (\mathbb 1\otimes \Lambda(\phi_0))(\Lambda(\phi_0)\otimes \mathbb 1)((a,x)\otimes (1,0)\otimes \phi_2(y\otimes z)) \\
			&& + \mathbb 1\otimes \Lambda(\phi_0))(\Lambda(\phi_2)\otimes \mathbb 1)((a,x)\otimes (c,z)\otimes (b,y)\\
			&& + (a,x)\otimes (1,0) \otimes (0,[y,z]))\\
			&& + \mathbb 1\otimes \Lambda(\phi_2))(\Lambda(\phi_0)\otimes \mathbb 1)(\Lambda(\phi_2)\otimes \mathbb 1)((a,x)\otimes (c,z)\otimes (b,y)\\
			&& + (a,x)\otimes (1,0) \otimes (0,[y,z]))\\
			&& + (\mathbb 1\otimes \Lambda(\phi_0))(\Lambda(\phi_1)\otimes \mathbb 1)((a,x)\otimes (1,0) \otimes \phi_1(y\otimes z))\\
			&& + (\mathbb 1\otimes \Lambda(\phi_1))(\Lambda(\phi_0)\otimes \mathbb 1)((a,x)\otimes (1,0) \otimes \phi_1(y\otimes z))\\
			&&  (\mathbb 1\otimes \Lambda(\phi_1))(\Lambda(\phi_1)\otimes \mathbb 1)((a,x)\otimes (c,z)\otimes (b,y) \\
			&& + (a,x)\otimes (1,0) \otimes [y,z]) \\
			&=& (1,0)\otimes \phi_2(y\otimes z) \otimes (a,x) + (1,0)\otimes (1,0) \otimes [x,\phi_2(y\otimes z)]\\
			&& + (1,0)\otimes (b,y) \otimes \phi_2(x\otimes z) + (1,0) \otimes (1,0) \otimes [y,\phi_2(x\otimes z)]\\
			&& + (c,z) \otimes (1,0) \otimes \phi_2(x\otimes y) + (1,0) \otimes (1,0) \otimes \phi_2([x,z]\otimes y)\\
			&& + (1,0)\otimes (1,0) \otimes \phi_2(x\otimes [y,z]) \\
			&& + (1,0)\otimes (1,0) \otimes \phi_1(x\otimes \phi_1(y\otimes z)) \\
			&& + (1,0)\otimes (1,0) \otimes \phi_1(\phi_1(x\otimes z)\otimes y),
		\end{eqnarray*}
	where we have indicated terms of type $(0,x)$ for $x\in \frak g$ by $x$, for ease of notation. Similarly, for the LHS of the YB equation in degree $2$, which we indicate by $\Psi^L_2$, we find the equality
	\begin{eqnarray*}
		\Psi^L_2 (a,x)\otimes (b,y) \otimes (c,z) &=& (c,z) \otimes (1,0) \otimes \phi_2(x\otimes y) + (1,0) \otimes (1,0) \otimes [\phi_2(x\otimes y), z] \\
		&& + (1,0) \otimes (b,y) \otimes \phi_2(x\otimes z) + (1,0) \otimes \phi_2(y\otimes z) \otimes (a,x)\\
		&& + (1,0) \otimes (1,0) \otimes \phi_2([x,y]\otimes z)
		+ (1,0) \otimes (1,0) \otimes \phi_1(\phi_1(x\otimes y) \otimes z). 
	\end{eqnarray*}  
	Equating the two terms we find that the YB equation holds if and only if
	\begin{eqnarray*}
	\lefteqn{(1,0)\otimes (1,0) \otimes ([\phi_2(x\otimes y), z] + \phi_2([x,y]\otimes z) + \phi_1(\phi_1(x\otimes y)\otimes z))} \\
	&=& (1,0) \otimes (1,0) \otimes ([x,\phi_2(y\otimes z)] + [y,\phi_2(x\otimes z)] + \phi_2([x,z]\otimes y) + \phi_x(x\otimes [y,z])\\
	&& + \phi_1(x\otimes \phi_1(y\otimes z)) + \phi_1(\phi_1(x\otimes z)\otimes y)).
	\end{eqnarray*}
	Up to two tensorands $(1,0)$, we see that this equation is equivalent to
	$$
	\delta^2_{\rm Lie} \phi_{2} + \sum_{k=1}^m \frac{1}{2} [\phi_1,\phi_1] = 0,
	$$
	 which completes the case $m=1$. 
	 
	 We now suppose that the statement holds true for some $m>1$, and we want to verify it for $m+1$. We let $\phi_i$ be a family of Lie $2$-cochains such that $\hat R = \sum_{i=0}^m \Lambda(\phi_i)$ is a YB deformation of degree $m$, and we want to derive the obstruction for $\Lambda(\phi_{m+1})$, where $\phi_{m+1}$ is a Lie $2$-cochian, to give a deformation of degree $m+1$. From the assumptions, we just need to impose that the YBE holds for terms in degree $m+1$. We observe that when considering the terms of type 
	 $$
	 (\Lambda(\phi_i)\otimes \mathbb 1)(\mathbb 1\otimes \Lambda(\phi_j)) (\Lambda(\phi_k)\otimes \mathbb 1) - (\mathbb 1\otimes \Lambda(\phi_i)) (\Lambda(\phi_j)\otimes \mathbb 1)(\mathbb 1\otimes \Lambda(\phi_k)),
	 $$
	 with $i=m+1$, or $j=m+1$ or $k=m+1$, this gives us the Lie algebra $2$-cocycle condition for $\phi_{m+1}$ up to an overal tensor product of $(1,0)\otimes (1,0)$ as for the case with $m=1$. Therefore, these terms give rise to $\delta^2_{\rm Lie}(\phi_{m+1})$ of Equation~\ref{eqn:YB_obstruction}. Let us now consider the terms where more than one subscript of the $\phi_i$ is nontrivial. We distinguish four different cases, depending on which component of $\Gamma_{m+1}$ the triple $(i,j,k)$ belongs to. We consider first the terms $(\Lambda(\phi_i)\otimes \mathbb 1)(\mathbb 1\otimes \Lambda(\phi_j)) (\Lambda(\phi_k)\otimes \mathbb 1)$. When $(i,j,k)\in \Gamma_{m+1}^1$, we have
	 \begin{eqnarray*}
	 	\lefteqn{(\Lambda(\phi_i)\otimes \mathbb 1)(\mathbb 1\otimes \Lambda(\phi_j)) (\Lambda(\phi_k)\otimes \mathbb 1)((a,x)\otimes (b,y) \otimes (c,z))}\\
	 	 &=& (\Lambda(\phi_0)\otimes \mathbb 1)(\mathbb 1\otimes \Lambda(\phi_k)) (\Lambda(\phi_{m+1-k})\otimes \mathbb 1)((a,x)\otimes (b,y) \otimes (c,z))\\
	 	 &=& (1,0) \otimes (1,0) \otimes \phi_k(\phi_{m+1-k}(x\otimes y)\otimes z).
	 \end{eqnarray*}
	A direct inspection shows that when $(i,j,k) \in \Gamma_{m+1} - \Gamma_{m+1}^1$, the term $(\Lambda(\phi_i)\otimes \mathbb 1)(\mathbb 1\otimes \Lambda(\phi_j)) (\Lambda(\phi_k)\otimes \mathbb 1)((a,x)\otimes (b,y) \otimes (c,z))$ vanishes for all simple tensors. When considering the terms of type $(\mathbb 1\otimes \Lambda(\phi_i)) (\Lambda(\phi_j)\otimes \mathbb 1)(\mathbb 1\otimes \Lambda(\phi_k))$, we have that for $(i,j,k) \in \Gamma_{m+1}^1, \hat \Gamma_{m+1}$ the terms vanish identically, while for $(i,j,k) \in \Gamma_{m+1}^2$ we obtain $(1,0)\otimes (1,0) \otimes \phi_k(x\otimes \phi_{m+1-k}(y\otimes z))$, and for $(i,j,k) \in \Gamma_{m+1}^3$ we obtain $(1,0) \otimes (1,0) \otimes \phi_k(\phi_{m+1-k}(x\otimes z)\otimes y)$. Therefore, for each $k = 1, \ldots , m$ we obtain (up to a tensor factor) $\frac{1}{2} [\phi_k,\phi_{m+1-k}]$. Putting all the terms together completes the proof. 
\end{proof}

	A perturbative expansion of a YB operator is a deformed YB operator with higher order (i.e. at least quadratic) deformations.

	\begin{corollary}
		Let $\frak g$ be an $n$-Lie algebra, and let $R$ denote the associated YB operator. The obstruction to lifting a degree $k$ $\Lambda$-deformation to a degree $k+1$ $\Lambda$-deformation lies in the third Lie algebra cohomology group $H^3(\frak g, \frak g)$.  
	\end{corollary}
	\begin{proof}
			Given a degree $k$ $\Lambda$-deformation, from Theorem~\ref{thm:Higher_obstruction} we see that the obstruction coincides with the Lie algebra obstruction of degree $k+1$, which is well known to lie in $H^3(\frak g, \frak g)$ (see for instance \cite{NR} for the binary case, and \cite{Tak} for the $n$-ary case). 
	\end{proof}

	The following result, concerning the perturbative expansion of YB operators, is now immediate.
	\begin{corollary}\label{cor:perturbative_lambda}
		Let $\frak g$ be an $n$-Lie algebra, and let $R$ denote the associated operator. Assume that a nontrivial infinitesimal $\Lambda$-deformation of $R$ exists. Then, if $H^3(\frak g, \frak g) = 0$ we can deform $R$ arbitrarily many times.
	\end{corollary}

	The results above give a procedure to start with a Lie algebra $\frak g$, obtain a YB operator over the $\mathbb k$-module, and then produce a perturbative series $\hat R = \sum_{i=0}^\infty \hbar^i R_i$, where $R_0 = R$, that satisfies the YB equation over the $\mathbb k[[\hbar]]$-module $\hat X = \mathbb k[[\hbar]]\otimes X$. This can automatically be done whenever $\frak g$ has nontrivial second cohomology, and trivial third cohomology.

	However, from Theorem~\ref{thm:quandle_deform} we also obtain that when $\mathfrak g$ has trivial center, and $H^2_{\rm Lie}(\mathfrak g,\mathfrak g) = 0$, the corresponding YB operator does not admit $\Lambda$-deformations.
	
	\begin{corollary}
			Let $\mathfrak g$ be an $n$-Lie algebra with trivial center and $H^2_{\rm Lie}(\mathfrak g,\mathfrak g) = 0$. Let $R$ denote the corresponding YB operator. Then, $R$ does not admit  $\Lambda$-deformations.
	\end{corollary}
	\begin{proof}
		 Theorem~\ref{thm:quandle_deform} and the results of \cite{YB_deform} show that deforming the Lie algebra structure is equivalent to deforming the SD structure. The rigidity of $\mathfrak g$ implies that the SD structure is rigid as well. A direct computation shows that $\Lambda$-deformations are equivalent to deformations of the underlying SD structure, completing the proof.
	\end{proof}

\section{More on second cohomology}\label{sec:coh}

It is of interest to consider more in detail a study of the second cohomology group of YB operators. In fact, in order to be able to produce perturbative expansions (of any degree), one needs the second cohomology group to be nontrivial. 

In this section we assume that the ground field $\mathbb k$ is of zero characteristic, and $\frak g$ indicates a (binary) Lie algebra. In the following, we assume the convention, inspired by Sweedler's notation, that a map $\phi : A \longrightarrow B\otimes B$ is written as $\phi(x) = \phi(x)_1\otimes \phi(x)_2$, where a summation is intended.

\begin{lemma}\label{lem:2_cocy}
	Let $X = \mathbb k \oplus \frak g$, and let $R$ be the YB operator associated to $\frak g$, which is assumed to have trivial center  and to be perfect (e.g. it is semisimple). Suppose that $\phi$ is a YB $2$-cocycle. Then, $\phi$ is characterized by the following conditions:
	\begin{itemize}
		\item[(i)]
			$(\pi_0\otimes \pi_0)\phi : \frak g\otimes \frak g \longrightarrow \mathbb k$ is a Lie algebra $2$-cocycle with coefficients in $\mathbb k$.
		\item[(ii)] 
			$\phi(\mathbb k\otimes \mathbb k) = 0$. 
		\item[(iii)]	
			$(\pi_1\otimes \pi_0)\phi (\mathfrak g\otimes \mathfrak g) = 0$.
		\item[(iv)] 
			$\phi((1,0)\otimes (0,x)) = (1,0)\otimes g(x)$ and $\phi((0,x)\otimes (1,0)) = - (1,0) \otimes g(x)$ for some Lie algebra derivation $g: \frak g \longrightarrow \frak g$. 
		\item[(v)] 
			$\delta^2_{\rm Lie}((\pi_0\otimes \pi_1)\phi)(x\otimes y\otimes z) = [[x,\pi_1\phi(y\otimes z)_1],\pi_1\phi(y\otimes z)_2] - [g(x),[y,z]]$.
		\item[(vi)] It holds:
			\begin{eqnarray*}
				\begin{aligned}
					\lefteqn{\phi([x,y]\otimes z)_1 \otimes \phi([x,y]\otimes z)_2}\\
					 &=& \phi(y\otimes z)_1 \otimes  [x,\pi_1\phi(y\otimes z)_2] + \phi(x\otimes z)_1 \otimes  [\pi_1\phi(x\otimes z)_1,y].
				\end{aligned}
			\end{eqnarray*}
		\item[(vii)] It holds:
				\begin{eqnarray*}
					\begin{aligned}
						&&  [\pi_1\phi(x\otimes y)_1,z] \otimes \phi(x\otimes y)_2 + \phi(x\otimes y)_1\otimes [\pi_1\phi(x\otimes y),z] \\
						&& +   [y,\pi_1\phi(x\otimes z)_1]\otimes \phi(x\otimes z)_2 - 1 \otimes g(y) \otimes [x,z] + 1\otimes g(z) \otimes [x,y]\\
						&=& \phi(y\otimes z)_2 \otimes [x,\pi_1\phi(y\otimes z)_1] - [y,z]+ \otimes g(x)\\ 
						&& + \phi([x,z]\otimes y)_1 \otimes \phi([x,z]\otimes y)_2 + \phi(x\otimes [y,z])_1\otimes \phi(x\otimes [y,z])_2. 
					\end{aligned}
				\end{eqnarray*}
	\end{itemize}
\end{lemma}
\begin{proof}
	Let us set $\hat \phi = \phi_0 + \hbar \phi_1$, where $\phi_0 := R$ and $\phi_1 := \phi$. Then, $\phi$ being a $2$-cocycle means that we have the equality
	\begin{eqnarray}\label{eqn:YB_2cocy}
	\begin{aligned}
	\lefteqn{(\phi_1\otimes \mathbb 1)(\mathbb 1\otimes \phi_0)(\phi_0\otimes \mathbb 1) + 	(\phi_0\otimes \mathbb 1)(\mathbb 1\otimes \phi_1)(\phi_0\otimes \mathbb 1) +
		(\phi_0\otimes \mathbb 1)(\mathbb 1\otimes \phi_0)(\phi_1\otimes \mathbb 1)}\\
	&=&   (\mathbb 1\otimes \phi_1)(\phi_0\otimes \mathbb 1)(\mathbb 1\otimes \phi_0) + (\mathbb 1\otimes \phi_0)(\phi_1\otimes \mathbb 1)(\mathbb 1\otimes \phi_0) + (\mathbb 1\otimes \phi_0)(\phi_0\otimes \mathbb 1)(\mathbb 1\otimes \phi_1).
	\end{aligned}
	\end{eqnarray}
	The proof consists of a tedious direct analysis of the equality evaluated on different types of simple tensors in $X^{\otimes 3}$, where $X = \mathbb k \oplus \frak g$. Recall that we indicate by $\pi_0 : X \longrightarrow \mathbb k$ the projection on the first direct summand, and likewise by $\pi_1: X \longrightarrow \frak g$ the projection on the second summand. Also, recall the (Sweedler inspired) notation $\phi_i (u \otimes v)_{(1)} \otimes \phi_i (u \otimes v)_{(2)}$ to indicate the sum of terms in $X\otimes X$ in the image of $\phi_i$, for $i=0,1$.  To indicate the effect of comultiplication on the term two terms above, we will still utilize the Sweedler notation  for comultiplication, where a superscript is employed. We also use the shorthand $a + x$ for $(a,x)$.
	
	Observe that, in general, we have $\phi_1(1\otimes 1) = r\cdot 1\otimes 1 + 1\otimes v_1 + v_2\otimes 1 + u_i^{(1)}\otimes u_i^{(2)}$, for some $r \in \mathbb k$, some fixed vectors $v_1, v_2 \in \mathbb g$, and $u_i^{(1)}\otimes u_i^{(2)} \in \mathfrak g\otimes \mathfrak g$. 
	However, Equation~\eqref{eqn:YB_2cocy} on simple tensors of type $x\otimes 1\otimes 1$ and $1 \otimes 1 \otimes z$ gives that all terms in $\phi_1(1\otimes 1)$ are zero, except possibly for $r\cdot 1\otimes 1$, where we use the fact that $\mathfrak g$ has trivial center. So, we have $\phi_1(1\otimes 1) = r\cdot 1\otimes 1$. 
	
	We evaluate Equation~\eqref{eqn:YB_2cocy} on tensors of type $x\otimes 1\otimes z$, giving us the equation
	\begin{eqnarray*}
				\lefteqn{r\cdot 1\otimes 1 \otimes [x,z] + 1\otimes [\pi_1\phi_1(x\otimes 1)_1,z] \otimes \phi_1(x\otimes 1)_2 + 1\otimes \phi_1(x\otimes 1)_1\otimes [\pi_1\phi_1(x\otimes 1)_2,z]}\\
				&=& 1\otimes \phi_1([x,z]\otimes 1)_1\otimes \phi_1([x,z]\otimes 1)_2 + \phi_1(1,\otimes z)_1 \otimes 1 \otimes [x,\pi_1\phi_1(1\otimes z)_2] \\
				&& + 1\otimes \phi_1(1\otimes z)_2 \otimes [x,\pi_1\phi_1(1\otimes z)_1] + 1\otimes 1 \otimes [[x,\pi_1\phi_1(1\otimes z)_1],\pi_1(1\otimes z)_2],
	\end{eqnarray*}
	which forces $(\pi_1\otimes \pi_1)\phi_1(1\otimes z) = 0$ for all $z\in \mathfrak g$. 
	
	Equation~\eqref{eqn:YB_2cocy} on tensors of type $x\otimes y\otimes 1$ produces a term in the equation of type $\phi_1([x.y]\otimes 1)_1 \otimes 1 \otimes \phi_1([x,y]\otimes 1)_2$ which cannot have components in $\mathfrak g \otimes \mathbb k \otimes \mathbb k$ because it cannot be balanced by other terms. Therefore, using the fact that $\mathfrak g$ is perfect, it follows that $(\pi_1\otimes \pi_0)\phi_1(x\otimes 1) = 0$ for all $x$ in $\mathfrak g$. Similarly, the component $\mathbb k \otimes \mathbb k \otimes k$ of this equation gives that $(\pi_0\otimes \pi_0)\phi_1(x\otimes 1) = 0$ for all $x$ in $\mathfrak g$. 
	
	Equation~\eqref{eqn:YB_2cocy} on tensors of type $1\otimes y\otimes z$ projected on the components $\mathbb k \otimes \mathbb k \otimes \mathfrak g$ gives
	$$
			1\otimes \pi_0 \phi_1(1\otimes y)_1 \otimes [\pi_1\phi_1(1\otimes y)_2] = \pi_0\phi_1(1\otimes z)_1 \otimes 1 \otimes [\pi_1\phi_1(1\otimes z)_2,y] + 1\otimes \pi_0\phi_1(1\otimes [y,z])_1 \otimes \pi_1\phi_1(1\otimes [y,z])_2,
	$$
	from which we derive that $(\pi_0\otimes \pi_1)\phi_1(1\otimes x) = 1\otimes g(x)$ for some derivation $g : \mathfrak g \longrightarrow \mathfrak g$ of $\mathfrak g$. This is one of the equation characterizing $\phi_1$ in the statement of the lemma.
	Terms projected in $\mathbb k \otimes \mathfrak g \otimes \mathbb k$ give the symmetry
	$$
			[f(y),z] = r\cdot [y,z] + f([y,z]), 
	$$
	for all $y,z$ in $\mathfrak g$, where $f$ is defined through $(\pi_1\otimes \pi_0)\phi_1(1\otimes x) = f(x)\otimes 1$. Moreover, projecting on $\mathbb k \otimes \mathbb k \otimes \mathbb k$ (and using the fact that $\mathfrak g$ is perfect) we find that $(\pi_0\otimes \pi_0)\phi_1(1\otimes x) = 0$ for all $x\in \mathfrak g$. 
	
	Writing the symmetries that we have found up to now more explicitly, we can write $\phi_1(1\otimes x) = f_1(x)\otimes 1 + 1\otimes g_1(x)$, $\phi_1(x\otimes 1) = 1\otimes g_2(x) + h(x)_1\otimes h(x)_2$ and $\phi_1(1\otimes 1) = r\cdot 1\otimes 1$, where $g_1$ is a Lie algebra derivation. 
	
	Equation~\eqref{eqn:YB_2cocy} evaluated on $x\otimes y\otimes z$ projected on $\mathfrak g\otimes \mathfrak g\otimes \mathfrak g$ gives
	$$
			h(y)_1\otimes h(y_2)\otimes [x,z] = h(x)_1\otimes [y,z] \otimes h(x)_2,
	$$
	which taking $x=z$ forces $h(x)_1\otimes h(x)_2$ to be zero. This further simplifies the expression of $\phi_1(x\otimes 1)$. 
	
	Substituting Equation~\eqref{eqn:YB_2cocy} evaluated on $x\otimes z \otimes 1$ projected on $\mathbb k \otimes \mathbb k \otimes \mathfrak g$ into the equation on $x\otimes 1\otimes z$ projected on $\mathbb k \otimes \mathbb k \otimes \mathfrak g$ we get
	$$
			2r[x,z] = [x,g_2(z)] + [x,g_1(z)] + [x,f_1(z)],
	$$
	which gives $[x,2rz + g_2(z) - g_1(x) + f_1(z)] = 0$ for all $x$ and $z$. Since $\mathfrak g$ has trivial center, we get $2rz + g_2(z) - g_1(x) + f_1(z) = 0$. A similar approach also gives that $2ry + g_1(y) - g_2(y) = 0$ for all $y\in \mathfrak g$. Therefore we must have that $f_1(x) = -4rx$ for all $x$. However, from Equation~\eqref{eqn:YB_2cocy} evaluated on $1\otimes x\otimes y$ and projected on $\mathbb k \otimes \mathbb k \otimes \mathbb \mathfrak g$ we obtain that $[f_1(x), y] = r[x,y] + f_1([x,y])$. Substituting the $f_1$ just obtained we find that $r[x,y] = 0$ for all $x,y$, which is possible only if $r=0$, since $\mathfrak g$ is not abelian. This gives us that $\phi_1(1\otimes 1) = 0$, as in the statement of the lemma. Moreover, we also have that $g_2 = -g_1$. This completes the proof of facts (ii), (iii) and (iv).
	
	The proof of (i), (vi) and (vii) is obtained by considering Equation~\eqref{eqn:YB_2cocy} evaluated on $x\otimes y \otimes z$ projected on the  direct summands of $X\otimes X\otimes X$. In fact, projecting over $\mathbb k \otimes \mathbb k \otimes \mathbb k$ we obtain the equation 
	$$
			\alpha([x,y]\otimes z) = \alpha([x,z]\otimes y) + \alpha(x\otimes [y,z]),
	$$
	where $\alpha : \mathfrak g\otimes \mathfrak g \longrightarrow \mathbb k$ is given by $\alpha(x\otimes y) := (\pi_0\otimes \pi_0)\phi_1(x\otimes y)$. This equation is the $2$-cocycle condition for $\alpha$ with coefficients in $\mathbb k$, which gives us (i). Equation~\eqref{eqn:YB_2cocy} projected on $\mathfrak g\otimes \mathfrak g\otimes \mathfrak g$ was already considered above. The projections over $\mathbb k \otimes \mathfrak g\otimes \mathbb k$, and $\mathfrak g\otimes \mathfrak g \otimes \mathbb k$ are seen to be satisfied identically. The projection over $\mathfrak g\otimes \mathbb k \otimes \mathfrak g$ gives (vi), while the projection over $\mathbb k \otimes \mathfrak g\otimes \mathfrak g$ gives (vii). Finally, (v) follows from the projection on the summand $\mathbb k \otimes \mathbb k \otimes \mathfrak g$. 
\end{proof}

Let $\mathfrak g$ be a Lie algebra satisfying the hypotheses of Lemma~\ref{lem:2_cocy}. We define the groups $\mathcal Z(\mathfrak g)$, $\mathcal B(\mathfrak g)$ and $\mathcal H(\mathfrak g)$ as follows. The group $\mathcal Z(\mathfrak g)$ is defined as the group of triples $(g,\zeta, \xi)$ with $g: \mathfrak g\longrightarrow \mathfrak g$, $\zeta: \mathfrak g\otimes \mathfrak g\longrightarrow \mathfrak g$, and $\xi: \mathfrak g\otimes \mathfrak g\longrightarrow \mathfrak g\otimes \mathfrak g$ where  $g$ is a derivation, $\xi$ satisfies (vi), and the compatibility conditions (v) and (vii) between $g$, $\zeta$ and $\xi$ are satisfied. The group $\mathcal B(\mathfrak g)$ is defined as the subgroup of $\mathcal Z(\mathfrak g)$ where $g(x) = [w,x]$ (inner derivation), $\zeta(x\otimes y) = [h(x),y] + [x,h(y)] - h([x,y]) - s[x,y]$, and $\xi(x\otimes y) = -w\otimes [x,y]$, 
for some $s\in \mathbb k$, $w\in \mathfrak g$, and $h:\mathfrak g\longrightarrow \mathfrak g$. Finally, we set $\mathcal H(\mathfrak g) := \mathcal Z(\mathfrak g)/ \mathcal B(\mathfrak g)$. 

We give now a characterization of the second cohomology group of YB operators $R_{\frak g}$ arising from Lie algebras that are perfect, and with zero center (e.g. they are semisimple). 

\begin{theorem}\label{thm:second_coh}
		Let $\frak g$ be a perfect Lie algebra with zero center. Let us denote by $R$ the YB operator associated to it. Then, the second cohomology group of $R$ is given by $H^2_{\rm YB}(R) = H^2_{\rm Lie}(\mathfrak g, \mathbb k) \oplus \mathcal H(\mathfrak g)$.  
\end{theorem}
\begin{proof}
	The most difficult part of the proof is to obtain a characterization for the $2$-cocycles. This has been done in Lemma~\ref{lem:2_cocy}. They constitute the set $\mathcal Z(\mathfrak g)$ defined above. 
	
	We need to derive the coboundaries.
	The coboundary of a $1$-cochain $f: X \longrightarrow X$ is given by:
	\begin{eqnarray*}
			\delta^1_{\rm YB}(f)(a,x)\otimes (b,y) &=& a\cdot (1,0) \otimes (0, [w,y]) + b\cdot (1,0) \otimes (0,[x,w]) + (1,0) \otimes (0,[f_1(x),y])\\
			&&+  (1,0) \otimes (0,[x,f_1(y)]) - s\cdot (1,0)\otimes (0,[x,y]) - (0,w)\otimes (0,[x,y])\\
			&& - (1,0) \otimes (f_0([x,y],0)) - (1,0)\otimes (0,f_1([x,y])),
	\end{eqnarray*}
	where we have used the decomposition of $f$ as $f(a,x) = af(1) + f(x)$, we have set $f(1) = (s, w)$ and written $f(x) = (f_0(x), f_1(x))$, with $f_0 : X \longrightarrow \mathbb k$ and $f_1: X\longrightarrow \frak g$. While not immediately obvious, one can directly verify that these coboundaries satisfy (i) - (vii) in Lemma~\ref{lem:2_cocy} as required. 
	
	To complete the proof, we need now to consider the components of Equation~\ref{eqn:YB_2cocy} projected on the simple tensorands modulo the projections of the coboundaries. A direct (and rather tedious) analysis shows that projecting on summands other than $\mathbb k\otimes \mathbb k \otimes \mathbb k$ we obtain $\mathcal B(\mathfrak g)$ defined above. Combining this with the $2$-cocycles from Lemma~\ref{lem:2_cocy}, gives rise to the direct summand $\mathcal H(\mathfrak g)$. 
	
	Projecting on $\mathbb k\otimes \mathbb k \otimes \mathbb k$ we find $f_0([x,y])$, which is the Lie coboundary with trivial coefficients in $\mathbb k$. Observe that the $2$-cocycle component on $\mathbb k\otimes \mathbb k \otimes \mathbb k$ is precisely the $2$-cocycle condition for Lie cohomology with trivial coefficients, by Lemma~\ref{lem:2_cocy}. This term is completely independent of the triples in $\mathcal H(\mathfrak g)$. Therefore, we get the remaining summand $H^2_{\rm Lie}(\mathfrak g,\mathbb k)$. This completes the proof. 
\end{proof}

It is well known that Lie cohomology is trivial for semisimple Lie algebras. However, as we will see in the examples below, the cohomology of YB operators associated to semisimple Lie algebras is not necessarily trivial. In fact, it turns out that $\mathcal H(\mathfrak g) \neq 0$ when $\mathfrak g = {\mathfrak sl}_2(\mathbb C)$. This  fact is somewhat suprising, considering that semisimple Lie algebras do not admit any nontrivial deformations, which is the second Whitehead Lemma.

\begin{remark}
	{\rm 
			A Lie algebra $\frak g$ endowed with a $2$-cocycle $\beta: \frak g\otimes \frak g \longrightarrow \mathbb k$ is also called a {\it quasi-Frobenius Lie algebra}. Moreover, if $\beta$ is cobounded, $\mathfrak g$ is said to be a {\it Frobenius Lie algebra}. See \cite{Pham}. Therefore, the direct summand $H^2_{\rm Lie}(\mathfrak g, \mathbb k)$ of $H^2_{\rm YB}(R)$ is the set of equivalence classes of quasi-Forbenius structures on $\mathfrak g$. 
	}
\end{remark}

\begin{remark}
	{\rm 
			There is an interesting class of perfect Lie algebras with trivial center, namely the {\it sympathetic} Lie algebras of Benayadi \cite{Ben}. They are additionally assumed to satisfy the condition that all derivations are inner. Recent work of Burde and Wagemann has shown that sympathetic Lie algebras might have nontrivial second cohomology \cite{BW}, which therefore gives rise to nontrivial YB second cohomology. By virtue of the characterization in Theorem~\ref{thm:second_coh}, it is possible to attempt a complete study of $H^2_{\rm YB}(R)$ where $R$ is the YB operator associated to the $25$ dimensional Benayadi's Lie algebra $\frak g_B$, whose second Lie cohomology group is nontrivial. In fact, $H^2(\frak g_B, \mathbb C) = 0$, and all derivations are inner. So, the problem of determining $H^2_{\rm YB}(R)$ is simpler. 
	}
\end{remark}

\section{Examples}\label{sec:ex}

We now consider some examples of the theory developed in this article. In particular, we show that there exist YB operators that admit infinitely many deformations (i.e. their deformations are integrable), and show that starting from semisimple Lie algebras, we can find YB operators that have nontrivial deformations. 

We start by constructing a YB operator that can be deformed infinitely many times and, therefore, admits perturbative expansions of any order.

	\begin{example}
	{\rm 
		Let $\frak H_m$ be the Heisenberg Lie algebra of dimension $2m +1$. From \cite{Sant}, we know that the Betti numbers of $\frak H_m$ are given by 
		$$
		{\rm dim}H^p(\frak H_m, \frak H_m) = {2m\choose p} - {2m\choose p-2},
		$$
		where $p\leq m$, which is not restrictive due to Poincar\'e duality. Therefore, the $5$-dimensional Heisenberg Lie algebra $\frak H_2$ has
		$$
		{\rm dim}H^2(\frak H_2, \frak H_2) = 5 
		$$
		and
		$$
		{\rm dim}H^3(\frak H_2, \frak H_2) = 0.
		$$
		This shows, applying Corollary~\ref{cor:perturbative_lambda}, that the corresponding YB operator $R_{\frak H_2}$ can be deformed arbitrarily many times, giving rise to a perturbative expansion
		$$
		\hat R = \sum_{i=0}^\infty \hbar^i R_i,
		$$
		where $R_0 = R_{\frak H_2}$ is the original YB operator, and $R_1$ is any choice of $2$-cocycle in any class of $H^2(\frak H_2, \frak H_2)$. 
		
		More generally, nilpotent Lie algebras are known to have highly nontrivial cohomologies (in particular second cohomology), which have been studied in some special cases quite in detail. The same procedure can be applied whenever the third cohomology vanishes. Alternatively, one can consider the obstruction in third cohomology and determine whether this vanishes even when $H^3(\mathfrak g, \mathfrak g)$ is nonzero on a case by case basis.
	}
\end{example}

Let us now consider the semisimple case. For such a Lie algebra, we know that the bracket cannot be deformed. However, as we will see, there are nontrivial YB deformations, that therefore do not arise from Lie algebra deformations.

\begin{example}\label{ex:sl2}
	{\rm 
				Let $\mathfrak g := {\mathfrak sl}_2(\mathbb C)$ be the special linear Lie algebra of dimension $3$ with complex coefficients. It is a well known fact that $\mathfrak g$ has trivial cohomology, since it is simple (Whiteheads lemmas). A natural question that arises is whether the YB second cohomology of the operator associated to $\mathfrak g$ is trivial as well. In such a case, the operator could not be deformed and no perturbative expansion would exist. However, it turns out that the rigidity of the Lie algebra structure does not implies the rigidity of the corresponding YB operator. In fact, a direct computation using the characterization of Theorem~\ref{thm:second_coh} gives $\dim H^2_{\rm YB}(R) = 2$.
				
				 Since, following Theorem~\ref{thm:Higher_obstruction}, nontrivial $\Lambda$-deformations are equivalent to Lie algebra deformations, it follows that the deformations of the YB operator $R$ must not be $\Lambda$-deformations, and they are therefore not arising from deformations of the Lie algebra structure.  
	}
\end{example}

\appendix

\end{document}